\documentclass[12pt]{article}%
\usepackage{amsmath}
\usepackage{amssymb}
\usepackage{amsfonts}
\usepackage{graphicx}%
\providecommand{\U}[1]{\protect\rule{.1in}{.1in}}
\newtheorem{theorem}{Theorem}[section]

\newtheorem{corollary}[theorem]{Corollary}

\newtheorem{example}[theorem]{Example}

\newtheorem{lemma}[theorem]{Lemma}

\newtheorem{proposition}[theorem]{Proposition}
\newtheorem{remark}[theorem]{Remark}

\newenvironment{proof}[1][Proof]{\textbf{#1.} }{\ \rule{0.5em}{0.5em}}
\newdimen\dummy
\dummy=\oddsidemargin
\ifx\pdfoutput\relax\let\pdfoutput=\undefined\fi
\newcount\msipdfoutput
\ifx\pdfoutput\undefined\else
\ifcase\pdfoutput\else
\ifx\paperwidth\undefined\else
\ifdim\paperheight=0pt\relax\else\pdfpageheight\paperheight\fi
\ifdim\paperwidth=0pt\relax\else\pdfpagewidth\paperwidth\fi
\fi\fi\fi
\begin{document}

\title{Galois Coverings of Pointed Coalgebras}
\author{William Chin\\DePaul University, Chicago, Illinois 60610}
\maketitle

\begin{abstract}
We introduce the concept of a Galois covering of a pointed coalgebra. The
theory developed shows that Galois coverings of pointed coalgebras can be
concretely expressed by smash coproducts using the coaction of the
automorphism group of the covering. Thus the theory of Galois coverings is
seen to be equivalent to group gradings of coalgebras. An advantageous feature
of the coalgebra theory is that neither the grading group nor the quiver is
assumed finite in order to obtain a smash product coalgebra.

\end{abstract}

\section{Introduction}

Every coalgebra is equivalent to basic coalgebra, whose simple comodules
appear with multiplicity one. Over an algebraically closed field basic
colagebra is pointed and embeds in the path coalgebra of its quiver. This
motivates the study of pointed coalgebras as subcoalgebras of path coalgebras.
The theory of coverings of coalgebras extends the theory of coverings of
finite-dimensional algebras via quivers with relations.

The theory of Galois coverings of quivers with relations is a standard tool in
the representation theory of finite dimensional algebras see e.g.
\cite{BG81},\cite{Gabriel81},\cite{Riedtmann80}. In this paper we introduce
the concept of coverings of pointed coalgebras, based on Galois coverings of
quivers. The theory developed shows that coverings of coalgebras can be
concretely expressed by smash coproducts using the coaction of the
automorphism group of coverings. Thus the theory of Galois coalgebra coverings
is seen to be equivalent to group gradings of coalgebras. One advantageous
feature of the coalgebra theory is that neither the grading group nor the
quiver is assumed finite in order to obtain a smash product coalgebra. In the
case that the grading group is infinite we do not need to pass to coverings of
$\Bbbk$-categories, as done in \cite{CibilsMarcos06}.

We begin with a review of path coalgebras and embeddings of pointed coalgebras
into the path coalgebra of their quivers. \ Next we turn to the theory of
Galois coverings of quivers in Section \ref{QC1}. In this context, coverings
and smash coproducts are known in graph theory as derived graphs
\cite{GrossTucker}. A covering of a quiver is specified by an arrow weighting
(also known as a voltage assignment), where group elements are assigned to
edges. Coverings are obtained from a smash coproduct construction (i.e. the
derived graph with respect to a voltage assignment) on the base quiver. For
the associated path coalgebras, the covering path coalgebra becomes a smash
coproduct coalgebra with a canonical homomorphism onto the base path coalgebra
(Theorem \ref{CSM}). We wish to restrict to subcoalgebras of the path
coalgebras and we provide conditions on the covering that characterize when
this is possible in Theorem \ref{Cov}, using a designated normal subgroup of
the fundamental group of the base quiver. The conditions guarantee that the
covering coalgebra is a smash coproduct over the homogeneous base coalgebra,
which is graded by the automorphism group of the covering. Thus the covering
quiver has a subcoalgebra whose comodule category is equivalent to the
category of graded comodules over the base coalgebra. Moreover, the methods
here provide a construction of universal covering coalgebra which provides a
way of constructing all connected gradings of a pointed subcoalgebra of a path
coalgebra. As a converse to these results, in Theorem \ref{smash=cov} we show
that any connected grading of a pointed coalgebra gives rise to a smash
coproduct that serves as a covering coalgebra.

A grading of a pointed coalgebra is specified by an arrow weighting that comes
from a vertex lifting of a covering of the quiver. Different liftings yield
possibly different gradings and isomorphic smash coproducts. We show
explicitly in Section $\ref{LVI}$ how liftings and gradings are related are
related via vertex weightings, and how they are equivalent via smash coproduct
coalgebra isomorphisms.

We close with examples including single and double loop quivers, the Kronecker
quiver, Example \ref{Kron}, and the basic coalgebra of quantum $SL(2)$ at root
of unity, Example \ref{SL2}. In these examples we examine coverings the
gradability of arbitrary finite-dimensional comodules.

\bigskip

We partially follow the approach of E. Green in \cite{Green83} where coverings
of quivers with relations were studied for locally finite quivers with
relations. Green's paper contains a definition of a Galois covering for
quivers with relations, and establishes a universal object in the category of
such coverings. He also showed that the representation theory of finitely
generated group-graded algebras is essentially equivalent to the theory of the
representation theory of Galois coverings of finite quivers with relations.
More precisely, given a finitely generated algebra presented by a finite
quiver with relations, the category of representations of a covering with
Galois group $G$ is equivalent to the category of $G$-graded modules over the
algebra. Other related work concerning coverings of quivers with relations and
coverings of $\Bbbk$-categories includes \cite{Martinez83}%
,\cite{CibilsMarcos06},\cite{LeMeur07}. In this paper we work directly with
coverings via smash coproduct coalgebras, thereby obtaining results on graded
comodules. Thus we do not work directly with representations of quivers as is
done in \cite{Green83}.

In a manner dual to the phenomenon for quivers with relations and
finite-dimensional algebras, a pointed coalgebra $B$ can be embedded in $\Bbbk
Q$ in more than one way (see Section \ref{PC}) . However, for many
presentations of algebras by quivers with relations, there is a canonical
presentation which results in universal fundamental group \cite{LeMeur07}. It
would be interesting to extend this sort of result to coalgebras.

\bigskip

The reader may refer to \cite{Massey91},\cite{Hatcher} for basic information
on covering spaces and \cite{GrossTucker} for combinatorial versions of
covering graphs and smash coproduct quivers, known there as derived graphs of
voltage graphs. Basic coalgebra theory, including the theory of pointed
coalgebras and path coalgebras, is covered in e.g. \cite{Chin2004}. The
articles \cite{ChinMontg},\cite{CKQ2002},\cite{Chin2002},\cite{Simson01}%
,\cite{NowakSimson2002},\cite{Simson04},\cite{Simson05},\cite{Simson07}
contain results concerning path coalgebras.

\section{Path coalgebras\label{PC}}

The vertex set of a quiver $Q$ is denoted by $Q_{0}$ and the arrow set is
denoted by $Q_{1}$. For each arrow $a$ the start (or source) of $a$ is denoted
by $s(a)$ and the terminal (or target) of $a$ is written $t(a)$ where
$s,t:Q_{1}\rightarrow Q_{0}$ are functions. \textit{Paths} in $Q$ are, by
definition, finite concatenations of arrows, and are always directed. Paths
are written from right to left and are of the form $a_{n}a_{n-1}...a_{1}$ with
$a_{i}\in Q$ and $s(a_{i})=t(a_{i-1})$ for $i=2,..n$. Using formal inverses of
arrows, we also consider possibly nondirected paths $a_{n}^{\pm}a_{n-1}^{\pm
}...a_{1}^{\pm}$ with $s(a^{-})=t(a)$ and $t(a^{-})=s(a),$ referred to as
\textit{walks in} $Q.$ Here we are slightly abusing terminology since paths
and walks are in general not elements of $Q.$ The set of paths of length $n$
is denoted by $Q_{n}$ and their span by $\Bbbk Q_{n}.$ For each pair $x,y\in
Q_{0}$ we write $Q(x,y)$ (resp. $Q^{\pm}(x,y)$) for the set of paths (resp.
walks) starting at $x$ and ending at $y.$ Here we have $B(x,y)=\Bbbk
Q(x,y)\cap B$ and $B=\bigoplus\limits_{x,y\in Q_{0}}B(x,y).$

J. A. Green (see \cite{Chin2004}) showed that the structure theory for
finite-dimensional algebras carries over well to coalgebras, with injective
indecomposable comodules replacing projective indecomposables. Define the
(\textit{Gabriel- }or\textit{\ Ext-}) \textit{quiver} of a coalgebra $C$ to be
the directed graph $Q(C)$ with vertices $\mathcal{G}$ corresponding to
isomorphism classes of simple comodules and $\mathrm{dim}_{\mathrm{k}%
}\mathrm{Ext}^{1}(h,g)$ arrows from $h$ to $g$, for all $h,g\in\mathcal{G}$.
The blocks of $C$ are (the vertices of) components of the undirected version
of the graph $Q(C)$. In other words, the blocks are the equivalence classes of
the equivalence relation on $\mathcal{G}$ generated by arrows. The
indecomposable or \textquotedblleft block\textquotedblright\ coalgebras are
the direct sums of injective indecomposables having socles from a given block.
When $C$ is a pointed coalgebra we may identify $\mathcal{G}$ with the set of
group-like elements $G(C)$ of $C$, and we may take the arrows to be a basis of
a nonuniquely chosen $\Bbbk-$space $I_{1}$ spanned by nontrivial skew
primitive elements. Here $C_{0}\oplus I_{1}=C_{1}$.

The \textbf{path coalgebra} $\Bbbk Q$ of a quiver $Q$ is defined to be the
span of all paths in $Q$ with coalgebra structure
\[%
\begin{array}
[c]{c}%
\Delta(p)=\sum_{p=p_{2}p_{1}}p_{2}\otimes p_{1}+t(p)\otimes p+p\otimes s(p)\\
\varepsilon(p)=\delta_{|p|,0}%
\end{array}
\]
where $p_{2}p_{1}$ is the concatenation $a_{t}a_{t-1}...a_{s+1}a_{s}%
...a_{1}\,$of the subpaths $p_{2}=a_{t}a_{t-1}...a_{s+1}$ and $p_{1}%
=a_{s}...a_{1}$ ($a_{i}\in Q_{0}$)$.$ Here $|p|=t$ denotes the length of $p$
and the starting vertex of $a_{i+1}\,$ is required to be the end of $a_{i}.$
Thus vertices are group-like elements, and if $a$ is an arrow $g\leftarrow h$,
with $g,h\in Q_{0},$ then $a$ is a $(g,h)$- skew primitive, i.e., $\Delta
a=g\otimes a+a\otimes h.$ It follows that that $kQ$ is pointed with coradical
$(kQ)_{0}=kQ_{0}\,$and the degree one term of the coradical filtration is
$(kQ)_{1}=kQ_{0}\oplus kQ_{1}.$ Moreover have the coradical grading $\Bbbk
Q=\bigoplus\limits_{n\geq0}\Bbbk Q_{n}$ by path length. The path coalgebra may
be identified with the cotensor coalgebra $\oplus_{n\geq0}(\Bbbk Q_{1})^{\Box
n}$ of the $\Bbbk Q_{0}$-bicomodule $\Bbbk Q_{1},$ cf. \cite{Nichols78}.

Let $B$ be a pointed coalgebra over a field $\Bbbk$. Then $B$ embeds
nonuniquely as an admissible subcoalgebra of the path coalgebra $\Bbbk Q$ of
the Gabriel quiver associated to $B.$ We review this embedding here (see
\cite{ChinMontg}, \cite{Woodcock97}, or \cite{Simson04}; cf. \cite{Nichols78}%
). Write $B=kQ_{0}\oplus I$ for a (nonunique) coideal $I,$ with projection
$\pi_{0}:B\rightarrow B_{0}$ along $I$, and write $\Bbbk Q=\Bbbk Q_{0}\oplus
J$ where $J=\Bbbk Q_{1}\oplus\Bbbk Q_{1}\Box\Bbbk Q_{1}\oplus...$ and $\Bbbk
Q=\Bbbk Q_{0}\oplus J$ is the cotensor coalgebra over $\Bbbk Q_{0}$. We
identify $B_{0}$ with $\Bbbk Q_{0},$ and let $I_{1}=I\cap B_{1}$, which we
identify with $\Bbbk Q_{1}$. Define the embedding
\[
\theta:B\rightarrow kQ=kQ_{0}\oplus J
\]
by $\theta(d)=\pi_{0}(d)+\sum_{n\geq1}\pi_{1}^{\otimes n}\Delta_{n-1}(d)$ for
all $d\in B,$ where $\pi_{1}:\mathrm{\Bbbk}Q\rightarrow I_{1}$ is a $B_{0}%
$-bicomodule projection onto $I_{1}=\Bbbk Q_{1}$.

It is very well-known that an algebra might be presented by quivers with
relations in essentially different ways . Of course this happens
coalgebraically as well. For example, let $Q$ be the quiver
\[
z\overset{a}{\underset{b}{\leftleftarrows}}y\overset{c}{\leftarrow}z
\]
and consider the subcoalgebra of $\Bbbk Q$ spanned by the arrows and vertices
together with the degree two element $ac.$ If we replace $ac$ by $ac+bc$ we
obtain a different, but isomorphic, subcoalgebra. In the embedding above, the
nonuniqueness can be seen to be a result of the choice of the skew primitive
space $I_{1}.$ It will be apparent in Section\ \ref{GQC} that these
subcoalgebras are associated to different subgroups of fundamental group of
the quiver.

Henceforth we shall assume a fixed embedding of the pointed coalgebra $B$ into
its path coalgebra so that $B\subseteq\Bbbk Q$ is an admissible subcoalgebra,
i.e., $B$ contains the vertices and arrows of $Q.$ We shall always assume that
$B$ is indecomposable and, equivalently, that $Q$ is connected as a graph.

\section{Quivers and coverings\label{QC1}}

The quiver $Q$ may be realized as a topological space, momentarily dropping
the orientation of arrows, as in e.g. \cite{Massey91}, \cite{Hatcher}. Let
$F:\tilde{Q}\rightarrow Q$ be a topological Galois covering with base points
$\tilde{x}_{0},x_{0}\in Q_{0},$ $F(\tilde{x}_{0})=x_{0}$. All coverings are
assumed to be connected. The topological covering space $\tilde{Q}$ can be
realized as a quiver first by giving it a graph structure and then by
assigning an orientation consistent with the orientation on $Q.$ The vertices
of $\tilde{Q}$ are the union of the fibers of the the vertices of $Q$, and the
arrows are the liftings of the arrows of $Q.$ We can then consider $F$ to be a
morphism of quivers, as done in \cite{Green83}. Coverings can be viewed as
quiver maps purely combinatorially, cf. \cite{Martinez83},\cite{LeMeur07}. One
can discretely construct a universal cover as the quiver as the quiver whose
vertices are in bijection with equivalence (homotopy) classes of walks in $Q.$
Then any covering is isomorphic to an orbit quiver under the free action of a
subgroup $G$ of the fundamental group $\pi_{1}(Q,x_{0})$ of $Q$.

For the purposes of this paper we discretely define a covering of a quiver $Q$
to be a surjective morphism of connected quivers $F:\tilde{Q}\rightarrow Q$
such that for all $\tilde{x}\in\tilde{Q}_{0}$, $F$ restricts to bijections
between the set of arrows starting at $\tilde{x}$ and the set of arrows
starting at $F(\tilde{x}),$ and also between the set of arrows ending at
$\tilde{x}$ and the arrows ending at $F(\tilde{x}).$

Let $\tilde{Q}\overset{F}{\rightarrow}Q\overset{F^{\prime}}{\leftarrow
}Q^{\prime}$ be two coverings of $Q.$ A morphism $\theta:F\rightarrow
F^{\prime}$ in the category of coverings of $Q$ is a quiver map $\theta
:\tilde{Q}\rightarrow Q^{\prime}$ such that $F^{\prime}\theta=F$. We shall
refer to such maps $\theta$ as a \textit{covering morphisms of quivers}.

The fundamental group functor provides an injective group homomorphism
$F_{\ast}:\pi_{1}(\tilde{Q},\tilde{x})\rightarrow\pi_{1}(Q,x)$ for $\tilde
{x}\in\tilde{Q}_{0}$ lifting $x\in Q_{0}$ for any covering $F:\tilde
{Q}\rightarrow Q.$ The covering is said to be \textit{Galois} (also known as
\textit{regular }or\textit{ normal}) if $F_{\ast}(\pi_{1}(\tilde{Q},\tilde
{x}_{0}))$ is a normal subgroup of $\pi_{1}(Q,x_{0})$. In this case $\pi
_{1}(Q,x_{0})/F_{\ast}(\pi_{1}(\tilde{Q},\tilde{x}_{0}))$ is the automorphism
group of $F$, viewed either combinatorially as a group of quiver
automorphisms, or topologically as a group of covering space automorphisms.

A fundamental property of coverings is that each walk in $Q$ can be lifted
uniquely to a walk in $\tilde{Q}$ once the starting point in $\tilde{Q}$ is
specified. More specifically if $w$ is a walk starting at \ $x\in Q_{0}$ and
$F(\tilde{x})=x$, then there is a unique walk $\tilde{w}$ starting at
$\tilde{x}$ and lifting $w.$ If $w$ is path, then the lifting $\tilde{w}$ is a
path. We say that a function $L:Q_{0}\rightarrow\tilde{Q}_{0}$ is a
\textit{lifting} of $F$ if $FL=\mathrm{id}_{Q_{0}}.$ For a walk $w$ in $Q$, we
shall write $L(w)$ to denote the unique lifting of $w$ starting at $L(s(w)).$

For a group $G$, we say that a function $\delta:Q_{1}\rightarrow G$ is an
\textit{arrow weighting}. For later use we also define a \textit{vertex
weighting} to be a function $\gamma:Q_{0}\rightarrow G.$ The map $\delta$ may
be extended to all walks by setting
\[
\delta(a_{n}^{e_{n}}\cdots a_{2}^{e_{2}}a_{1}^{e_{1}})=\delta(a_{n})^{e_{n}%
}\cdots\delta(a_{2})^{e_{2}}\delta(a_{1})^{e_{1}},
\]
$e_{i}=\pm$. We shall henceforth identify arrow weightings and their extension
to walks or paths in this manner.

Next let $G=\pi_{1}(Q,x_{0}))/F_{\ast}(\pi_{1}(\tilde{Q},\tilde{x}_{0}))$ be
the automorphism group of the Galois covering $F.$ There is an\textit{\ }arrow
weighting $\delta_{L}:Q_{1}\rightarrow G$ as follows. Let $w\in Q^{\pm}(x,y)$.
Let $L(w)$ be the lifting of $w$ starting at $L(x)$ and ending at $L(y)^{g}$,
$g\in G$. We define the map by letting $\delta_{L}(w)=g.$ Clearly $\delta_{L}$
depends on the choice of lifting $L.$ Restricting to arrows, we obtain an
arrow weighting $\delta_{L}.$ Now let $w=a_{n}^{e_{n}}\cdots a_{2}{}^{e_{2}%
}a_{1}^{e_{1}}\in Q^{\pm}(x,z)$. It is true that $\delta_{L}(a_{n}^{e_{n}%
}\cdots a_{2}^{e_{2}}a_{1}^{e_{1}})=\delta_{L}(a_{n})^{e_{n}}\cdots\delta
_{L}(a_{2})^{e_{2}}\delta_{L}(a_{1})^{e_{1}},$ for if $w=rq,$ with $q\in
Q^{\pm}(x,y)$ and $r\in Q^{\pm}(y,z),$ then $L(q)\in\tilde{Q}^{\pm
}(L(x),L(y)^{\delta_{L}(q)})$ and $L(r)\in\tilde{Q}(L(y),L(z)^{\delta_{L}%
(r)});$ hence $L(r)^{\delta_{L}(q)}L(q)\in\tilde{Q}(L(x),L(z)^{\delta
_{L}(r)\delta_{L}(q)})$ is a lifting of $rq$ and we conclude that $\delta
_{L}(rq)=\delta_{L}(r)\delta_{L}(q).$ Thus the weighting $\delta_{L}$ respects
concatenation in agreement with the extension from arrows to walks at the
beginning of this paragraph.

An arrow weighting $\delta:Q_{1}\rightarrow G$ is said to be
\textit{connected} if for all $x,y\in Q_{0}$ and $g\in G,$ there exists a walk
$w$ in $Q$ from $x$ to $y$ such that $\delta(w)=g.$ Assume $\delta$ is
connected and let $Q\rtimes G$ be the quiver with underlying vertex set
$(Q\rtimes G)_{0}=Q_{0}\times G=\{x\rtimes g|x\in Q_{0},g\in G\}$ and arrow
set $(Q\rtimes G)_{1}=Q_{1}\times G=\{a\rtimes g|a\in Q_{1},g\in G\},$
declaring that $s(a\rtimes g)=s(a)\rtimes g$ and $t(a\rtimes g)=t(a)\rtimes
\delta(a)g$. This construction shall be called the \textit{smash coproduct}
\textit{quiver} of $Q$ and $G.$ Let $F:Q\rtimes G\rightarrow Q$ be defined by
$F(u\rtimes g)=u$ for $u\in Q_{0}\cup Q_{1}$, $g\in G,$ and note that $F$ is a
surjective quiver morphism. Also, $\delta$ is connected if and only if
$Q\rtimes G$ is connected as a graph. It is known that

\begin{proposition}
\label{Prop1} If $\delta:Q_{1}\rightarrow G$ is a connected arrow weighting,
the canonical map $F:Q\rtimes G$ $\rightarrow Q$ is a Galois covering with
automorphism group $G$.
\end{proposition}

\begin{proof}
By the hypothesis that $\delta$ is connected, it is immediate $Q\rtimes G$ is
connected as a graph. By the construction of $Q\rtimes G,$ each $y\in Q_{0}$
has fiber $\{y\rtimes g|g\in G\}$ and each subgraph $z\overset{b}{\leftarrow
}y\overset{a}{\leftarrow}x$ in $Q$ with arrows $a,b$ and vertices $x,y,z$
lifts to a subgraph
\[
z\rtimes\delta(b)g\overset{b\rtimes g}{\longleftarrow}y\rtimes g\overset
{a\rtimes\delta^{-1}(a)g}{\longleftarrow}x\rtimes\delta^{-1}(a)g
\]
for all $g\in G.$ This observation makes it evident that there is a bijection
between the set of arrows ending at (resp. starting at) $y$ and the arrows
ending at (resp. starting at) $y\rtimes g.$ Also it is clear that these arrows
sets are disjoint for different $g.$ It follows that $F:Q\rtimes G$
$\rightarrow Q$ is a covering map.

We note that if $[w]\in\pi_{1}(Q,x_{0})$ is a walk class with closed walk $w,$
then the map $\pi_{1}(Q,x_{0})\rightarrow G$ defined by $[w]\mapsto\delta(w)$
is onto since $\delta$ is connected. Its kernel is the set of walk classes of
walks that lift to closed walks in $Q\rtimes G,$ i.e. the normal subgroup
$F_{\ast}(\pi_{1}(Q\rtimes G)).$ Thus the covering $F$ is Galois and the
automorphism group of the covering is $G=\pi_{1}(Q,x_{0})/F_{\ast}(\pi
_{1}(Q\rtimes G))$.
\end{proof}

\bigskip

Alternatively one can observe that the right action of $G$ on $Q\rtimes G$ is
free and that the orbit quiver is isomorphic to $Q.$

\section{Graded coalgebras from coverings\label{GQC}}

A coalgebra $C$ is said to be graded by the group $G$ if $C$ is the direct sum
of $\Bbbk$-subspaces $C=\oplus_{g\in G}C_{g}$ and
\[
\Delta(C_{g})\subseteq\sum_{\substack{ab=g\\a,b\in G}}C_{a}\otimes C_{b}%
\]
for all $g\in G$ and $\varepsilon(C_{g})=0$ for all $g\neq1.$ The element
$c\in C_{g}$ is said to be \textit{homogeneous of degree} $\delta(c)=g$ and we
shall adapt Sweedler notation and write $\Delta(c)=\sum c_{1}\otimes c_{2}$
always assuming homogeneous terms in the sum. In this case $C$ is a right
$\Bbbk G$-comodule coalgebra via the structure map $\rho:C\rightarrow
C\otimes\Bbbk G$, $\rho(c)=c\otimes g$ for all $g\in G$ and $c\in C_{g}$.
$\ $Conversely, every right $\Bbbk G$-comodule coalgebra is a $G$-graded coalgebra.

Let $C$ be a $G$-graded coalgebra. The smash coproduct coalgebra of $C$ and
$\Bbbk G$ is denoted by $C\rtimes\Bbbk G$ is defined as the $\Bbbk$-vector
space $C\otimes\Bbbk G$ with $\Delta:C\rtimes\Bbbk G\rightarrow(C\rtimes\Bbbk
G)\otimes(C\rtimes\Bbbk G)$ given by%
\[
\Delta(c\rtimes g)=\sum(c_{1}\rtimes\delta(c_{2})g)\otimes(c_{2}\rtimes g)
\]
for all homogeneous $c\in C$ and $g\in G.$ Let $F:C\rtimes\Bbbk G\rightarrow
C,$ $c\rtimes g\mapsto c$ be the canonical colagebra map onto $C.$ $G$ acts
canonically on the right as coalgebra automorphisms of $C\rtimes\Bbbk G$ via
$(c\rtimes g)^{h}=c\rtimes gh$ for all $h\in G.$ It is clear the action of $G$
preserves $F$, i.e. $F=F\circ$\thinspace$^{h}$.

\begin{proposition}
[\cite{DNRV}]The category of graded right comodules $\mathrm{Gr}^{C}$ of $C$
is equivalent the comodule category $\mathcal{M}^{C\rtimes\Bbbk G}$ over
$C\rtimes\Bbbk G$.
\end{proposition}

A graded comodule $M$ acquires the structure of a $C\rtimes\Bbbk G$-comodule
given by $\rho^{\prime}(m)=\sum m_{0}\otimes(m_{1}\rtimes\delta(m)^{-1}),$
$m\in M$. If $N$ is a $C\rtimes\Bbbk G$-comodule, then $N $ acquires a right
$\Bbbk G$-comodule structure via the coalgebra map $C\rtimes\Bbbk
G\rightarrow\Bbbk G,$ $c\rtimes g\mapsto\varepsilon(c)g^{-1}$; $N$ is a right
$C$-comodule via the coalgebra map $C\rtimes\Bbbk G\rightarrow\Bbbk G$,
$c\rtimes g\mapsto c$. Thus $N$ corresponds to an object of Gr$^{C}.$ See
\cite{DNRV} for more details.

\bigskip

Consider a Galois covering $F:\tilde{Q}\rightarrow Q$ with automorphism group
$G$ and lifting $L$. We saw in the previous section that there is an arrow
weighting associated to $L.$ The path coalgebra $\Bbbk Q$ is similarly a
$G$-graded coalgebra as follows: Let $p$ be a path in $Q(x,y)$. Let
$L(p)=\tilde{p}$ be the lifting of $p$ starting at $L(x)=\tilde{x}$ and ending
at $L(y)^{g}$. We define the degree map $\delta(p)=g.$ Suppose $p$ is the
concatenation of paths $p=rq.$ Then $L(p)=L(r)^{\delta(q)}L(q)$ so
$\delta(p)=\delta(r)\delta(q).$ It follows that $\delta$ determines a
coalgebra grading of $\Bbbk Q$ depending on the choice of lifting $L.$ It is
apparent that grading is determined by the arrow weighting obtained by
restricting to $Q_{1}.$ We accordingly form the smash coproduct $\Bbbk
Q\rtimes\Bbbk G=\Bbbk Q\rtimes_{L}\Bbbk G,$ with the canonical projection
$F_{L}:\Bbbk Q\rtimes_{L}\Bbbk G\rightarrow\Bbbk Q$, $F(p\rtimes g)=p$, using
this coaction of $\Bbbk G$ on $\Bbbk Q,$ sometimes leaving the lifting $L$ implicit.

\begin{theorem}
\label{CSM}Let $F:\tilde{Q}\rightarrow$ $Q$ a Galois covering and let
$L:Q_{0}\rightarrow\tilde{Q}_{0}$ be a lifting. There is a coalgebra
isomorphism $\psi:\Bbbk\tilde{Q}\rightarrow\Bbbk Q\rtimes\Bbbk G$ with
$F_{L}\circ\psi=F$ and a Galois covering isomorphism from $F_{L}:Q\rtimes
G\rightarrow Q$ to $F:\tilde{Q}\rightarrow Q$.
\end{theorem}

\begin{proof}
If $\tilde{p}$ is a path in $\tilde{Q},$ then $\tilde{p}$ is a lifting of
$F(\tilde{p})$. Hence there exists $\sigma(\tilde{p})\in G$ such that
$LF(\tilde{p})=\tilde{p}^{\sigma(\tilde{p})}.$ Since $\sigma(\tilde{p}%
)=\sigma(s(\tilde{p})),$ $\sigma=\sigma_{L}:\tilde{Q}_{0}\rightarrow G$ is a
vertex weighting on $\tilde{Q}$ that extends to a function on all paths in
$\tilde{Q},$ as specified. Let $g\in G$ and let $p\in Q(x,y)$. Define maps
$\phi:\Bbbk Q\rtimes\Bbbk G\rightarrow\Bbbk\tilde{Q}$ and $\psi:\Bbbk\tilde
{Q}\rightarrow\Bbbk Q\rtimes\Bbbk G$ by $\phi(p\rtimes g)=L(p)^{g}$ and
$\psi(\tilde{p})=F(\tilde{p})\rtimes\sigma(\tilde{p})$ for paths $p,\tilde{p}%
$. Then
\[
\Delta(p\rtimes g)=\sum_{p=rq}(r\rtimes\delta(q)g)\otimes(q\rtimes g)
\]
and
\[
\Delta(L(p)^{g})=\sum_{p=rq}L(r)^{\delta(q)g}\otimes L(q)^{g}%
\]
using the fact that $L(q)$ is the lifting starting at $L(x)$ and ending at
$L(y)^{\delta(q)}=s(L(r)^{\delta(q)}).$ It follows immediately that $\phi$ is
coalgebra map. Next observe that $\psi\phi(p\rtimes g)=\psi(L(p)^{g}%
)=F(L(p)^{g})\rtimes\sigma(L(p)^{g})=p\rtimes g.$ Similarly, $\phi
\psi=id_{\Bbbk\tilde{Q}}.$ Thus we have shown that $\phi$ and $\psi$ are
mutually inverse coalgebra isomorphisms. It is immediate that $F_{L}\psi=F.$
The isomorphism restricts to an isomorphism on vertices and arrows, so there
is a covering isomorphism $\tilde{Q}\cong Q\rtimes G$ as well.
\end{proof}

\bigskip

The result shows that the isomorphism type of the smash coproduct coalgebra
$\Bbbk Q\rtimes\Bbbk G$ does not depend on the choice of lifting $L,$ which is
determined by the grading $\delta_{L}$. In addition we have

\begin{proposition}
The smash coproduct coalgebra $\Bbbk Q\rtimes\Bbbk G$ is isomorphic to the
path coalgebra $\Bbbk(Q\rtimes G)$ of the smash coproduct quiver.
\end{proposition}

\begin{proof}
Define a map $E:\Bbbk Q\rtimes\Bbbk G\rightarrow\Bbbk(Q\rtimes G)$ on elements
$p\rtimes g$ where $p=a_{n}\cdots a_{2}a_{1}$ ($a_{i}\in Q_{1}$) is a path in
$Q$ and $g\in G$ by letting $E(p\rtimes g)$ be the concatenation%
\begin{align*}
&  (a_{n}\rtimes\delta(a_{n-1})\cdots\delta(a_{1})g)\cdots(a_{2}\rtimes
\delta(a_{1})g)(a_{1}\rtimes g)\\
&  =(a_{n}\rtimes\delta(a_{n-1}\cdots a_{1})g)\cdots(a_{2}\rtimes\delta
(a_{1})g)(a_{1}\rtimes g)
\end{align*}
noting that $E$ identifies the group-likes $x\rtimes g$ and the vertices
(denoted by the same symbol). Similarly $E$ identifies the skew-primitives
$a\rtimes g$ and the arrows. It is straightforward to check that $E$ is a
coalgebra isomorphism.
\end{proof}

\bigskip

We provide some terminology and notation:

\begin{itemize}
\item Let $B\subseteq\Bbbk Q$ be an admissible subcoalgebra. Let $b=\sum_{i\in
I}\lambda_{i}p_{i}\in B(x,y)$ with $x,y\in Q_{0}$ and distinct paths $p_{i}.$
We say that $b$ is a \textit{minimal element\ }of $B$ if $\sum_{i\in
I^{\prime}}\lambda_{i}p_{i}\notin B(x,y)$ for every nonempty proper subset
$I^{\prime}\subset I$, and $|I|\geq2.$ Clearly every element of $B$ is a
linear combination of paths and minimal elements.

\item Fix a base vertex $x_{0}\in Q_{0}$. We define a symmetric relation
$\sim$ on paths by declaring $p\sim q$ if there is a minimal element
$b=\sum_{i\in I}\lambda_{i}p_{i}\in B(x,y)$ where the $p_{i}$ are distinct
paths, $\lambda_{i}\in\Bbbk$, $x,y\in Q_{0}$ and $p=p_{1}$, $q=p_{2}.$ We
define $N(B,x_{0})$ to be the \textit{normal} subgroup of $\pi_{1}(B,x_{0})$
generated by equivalence (homotopy) classes of walks $w^{-1}p^{-1}qw$ where
$p,q$ are paths in $Q(x,y)$ with $p\sim q$ and $w$ is a walk from $x_{0}$ to
$x.$

\item Consider a Galois covering $F:\tilde{Q}\rightarrow Q$ with automorphism
group $G$ and lifting $L$. For each minimal element $b=\sum\lambda_{i}p_{i}\in
B$ we put $L(b)=\sum\lambda_{i}L(p_{i})$ and we let $\tilde{B}$ denote the
$\Bbbk$-span of $\{L(b)|L$ a lifting, $b\in B$ a minimal element or a path\}.
We say that the restriction $F:\tilde{B}\rightarrow B$ is a Galois
\textit{coalgebra covering} if every minimal element of $B$ can be lifted to
$\tilde{B}$ in the following sense: for every minimal element \ $b\in B(x,y)$
with $x,y\in Q_{0}$ and $\tilde{x}\in\tilde{Q}_{0},$ there exists $\tilde
{y}\in\tilde{Q}_{0}$ and a minimal element $\tilde{b}\in\tilde{B}(\tilde
{x},\tilde{y})$ such that $F(\tilde{b})=b.$
\end{itemize}

\begin{proposition}
\label{min}Let $F:\tilde{B}\rightarrow B$ be a Galois coalgebra covering. Then
\newline(a) $F(\min(\tilde{B}))=\min(B)$\newline(b) $F_{\ast}(N(\tilde
{B},\tilde{x}_{0}))=N(B,x_{0})$ for all $x_{0}\in Q,$ $\tilde{x}_{0}\in
\tilde{Q}$ with $F(\tilde{x}_{0})=x_{0}.$
\end{proposition}

\begin{proof}
Since each path in $Q$ lifts to a unique path in $\tilde{Q}$ starting at
$\tilde{x},$ it follows that each minimal element in $b\in B$ can be lifted
uniquely to an element of $\tilde{B}(\tilde{x},\tilde{y})$ starting at
$\tilde{x}$ and ending at $\tilde{y}$ for some $\tilde{y}\in\tilde{Q}_{0}.$
Now letting $\tilde{x}$ vary over $F^{-1}(x),$ we see that $F^{-1}(b)$ is the
consists of the set of liftings, one for each $\tilde{x}.$ It is immediate
that each such lifting is minimal in $\tilde{B}.$ Conversely, if $\tilde{b}%
\in\tilde{B}(\tilde{x},\tilde{y})$ is minimal, then it is the unique lifting
of $F(\tilde{b})$ starting at $\tilde{x};$ it follows that $F(\tilde{b})$ is
minimal. The conclusions follow.
\end{proof}

\bigskip

The fundamental example is given as follows. Let $B\subseteq\Bbbk Q$ be a
homogenous admissible subcoalgebra with respect to the grading given by an
arrow weighting $\delta:Q_{1}\rightarrow G$. The grading is said to be
connected if the arrow weighting is connected. If $b=\sum_{i\in I}\lambda
_{i}p_{i}\in B(x,y)$ is a minimal element, then it is necessarily homogeneous.
Consider the canonical map $F:\Bbbk Q\rtimes G\rightarrow\Bbbk Q$ defined by
$F(p\rtimes g)=p$ and consider the restriction to $B\rtimes\Bbbk G\rightarrow
B.$ Then under the identification of $\Bbbk Q\rtimes\Bbbk G$ with $\Bbbk
\tilde{Q}$ we easily see that $\tilde{B}=B\rtimes\Bbbk G.$ The liftings of
minimal element $b\in B$ are given by $b\rtimes g$ with $g\in G.$

\begin{theorem}
\label{Cov}The following are equivalent for a subcoalgebra $B\subseteq\Bbbk Q$
and Galois quiver covering $F:\tilde{Q}\rightarrow Q.$\newline(a) $B$ is a
homogeneous subcoalgebra of $\Bbbk Q.$\newline(b) $N(B,x_{0})\subseteq
F_{\ast}(\pi_{1}(\tilde{Q},\tilde{x}_{0}))$ for all $x_{0}\in Q,$ $\tilde
{x}_{0}\in\tilde{Q}$ with $F(\tilde{x}_{0})=x_{0}.$\newline(c) $F:\tilde
{B}\rightarrow B$ is a Galois coalgebra covering.\newline(d) $B$ is a
homogenous subcoalgebra of $\Bbbk Q$ and the grading is connected.
\end{theorem}

\begin{proof}
(a) implies (b). Suppose that $B$ is a graded subcoalgebra of $\Bbbk Q.$ Let
$[w^{-1}q^{-1}pw]$ be a generating element of $N(B,x)$ with $p\sim q.$ This
means that there are distinct paths $p,q\in Q(x,y)\ $with $b=p+\lambda
q+....\in B(x,y)$ minimal. The minimality of $b$ of forces it to be
homogeneous in the grading determined by $L.$ Accordingly, we have
$t(L(p))=t(L(q))=L(y)^{\delta(p)}=L(y)^{\delta(q)}.$ Now it is easy to check
that
\[
L(w)^{-1}(L(q)^{\delta(w)})^{-1}L(p)^{\delta(w)}L(w)
\]
is a closed path in $\tilde{Q}$ starting at $L(x_{0})$ lifting $w^{-1}%
q^{-1}pw.$ This shows that $[w^{-1}q^{-1}pw]\in F_{\ast}(\pi_{1}(\tilde
{Q},x_{0})).$

Assume as in (b) that $N(B,x_{0})\subseteq F_{\ast}(\pi_{1}(\tilde{Q}%
,\tilde{x}_{0}))$ and let $b=\sum_{i\in I}\lambda_{i}p_{i}\in B(x,y)$ with
$x,y\in Q_{0}$ and distinct paths $p_{i}$ be a minimal element of $B$. For
each $i\in I$, let $\tilde{p}_{i}$ be a lifting of $p_{i}$ starting at
$\tilde{x}=L(x)$ and ending at, say, $\tilde{y}_{i}$. Our assumption implies
that $N(B,y)\subseteq F_{\ast}(\pi_{1}(\tilde{Q},\tilde{y}_{1}))$ by the a
standard isomorphism (given by conjugation by a walk class from $y$ to $x_{0}%
$)$.$ Observe that $\tilde{p}_{2}\tilde{p}_{1}^{-1}$ is a walk from $\tilde
{y}_{1}$ to $\tilde{y}_{2}$ lifting the closed walk $p_{2}p_{1}^{-1}.$
Therefore $[p_{2}p_{1}^{-1}]\in N(B,y)\subseteq F_{\ast}(\pi_{1}(\tilde
{Q},\tilde{y}_{1})),$ and we see that $p_{2}p_{1}^{-1}$ also has a lifting
that is a closed walk starting and ending at $\tilde{y}_{1}$ Therefore
$\tilde{p}_{2}\tilde{p}_{1}^{-1}$ is a closed walk by e.g. \cite[Ch. 8, Lemma
3.3]{Massey91} and $\tilde{y}_{1}=\tilde{y}_{2}.$ This argument shows that the
$\tilde{y}_{i}$ are all equal. Thus $\tilde{b}=\sum_{i\in I}\lambda_{i}%
\tilde{p}_{i}$ is a lifting of $b$ in $B(\tilde{x},\tilde{y}_{1})$ and it is
easily seen to be minimal. This proves (b) implies (c).

Assume (c) and let $b=\sum_{i\in I}\lambda_{i}p_{i}\in B(x,y)$ with $x,y\in
Q_{0}$ be a minimal element of $B$. For each $i\in I$, let $\tilde{p}_{i}$ be
a lifting of $p_{i}$ starting at $\tilde{x}=L(x)$ and, by definition of the
grading, ending at $L(y_{i})^{\delta(p_{i})}$. The assumption forces the
$\delta(p_{i})$ to all be equal. This shows that $b$ is homogenous and thus
that $B$ is a graded subcoalgebra of $\Bbbk Q.$

We have shown that (a)-(c) are equivalent. We complete the proof by showing
that any grading of $B$ is connected. Let $x,y\in Q_{0}$ and let $g\in G.$
Then there exists a walk $\tilde{w}\in\tilde{Q}(L(x),L(y)^{g}).$ So
$w=F(\tilde{w})$ is a walk from $x$ to $y$ with lifting $\tilde{w}$, and
evidently $\delta(w)=g.$ Thus the grading is connected.
\end{proof}

\bigskip

Let $\tilde{Q}\overset{F}{\rightarrow}Q\overset{F^{\prime}}{\leftarrow
}Q^{\prime}$\ be two coverings of $Q.$\ Recall that a morphism $F\rightarrow
F^{\prime}$\ in the category of coverings of $Q$\ is given by a quiver
morphism $\theta:\tilde{Q}\rightarrow Q^{\prime}$\ such that $F^{\prime}%
\theta=F$. Consider coalgebra coverings of $B\subset\Bbbk Q$ arising from the
quiver coverings
\[%
\begin{array}
[c]{ccccc}%
\Bbbk\tilde{Q} & \overset{F}{\longrightarrow} & \Bbbk Q & \overset{F^{\prime}%
}{\longleftarrow} & \Bbbk Q^{\prime}\\
\cup &  & \cup &  & \cup\\
\tilde{B} & \longrightarrow & B & \longleftarrow & B^{\prime}%
\end{array}
\]
where the vertical maps are the presumed inclusions of admissible
subcolagebras. A \textit{morphism of coalgebra coverings} $(F:\tilde
{B}\rightarrow B)\rightarrow(F^{\prime}:B^{\prime}\rightarrow B)$ is a
morphism of quiver coverings $\theta$ as above such that the $\theta$
restricts to coalgebra map $\tilde{B}\rightarrow B^{\prime}$ (again abusively
denoting both the map on path coalgebras and its restriction by $\theta).$

\begin{corollary}
Assume any of the equivalent conditions of the Theorem hold and fix a lifting
$L:Q_{0}\rightarrow\tilde{Q}_{0}$. Then the coalgebra covering $F:\tilde
{B}\rightarrow B$ is isomorphic to $F_{L}:B\rtimes$ $G\rightarrow B.$
\end{corollary}

\begin{proof}
It is easy to check that $B\rtimes$ $G$ is the span of all liftings to
$Q\rtimes G$ of paths and minimal elements of $B.$ The result follows from
Theorem \ref{Cov}, which says that a lifting $L$ gives rise to a covering
isomorphism from $F_{L}:Q\rtimes G\rightarrow Q$ to $F:\tilde{Q}\rightarrow
Q.$
\end{proof}

\begin{lemma}
\label{morph}Let $\theta:\tilde{B}\rightarrow B^{\prime}$ be a morphism of
Galois coalgebra coverings of $B$. Then the map $\theta:\tilde{B}\rightarrow
B^{\prime}$ is a Galois coalgebra covering.
\end{lemma}

\begin{proof}
Adopt the notation of the preceding paragraph. Fix $x_{0}\in Q,$ $\tilde
{x}_{0}\in\tilde{Q},$ $x_{0}^{\prime}\in Q_{0}^{\prime}$ with $F(\tilde{x}%
_{0})=x_{0}=F^{\prime}(x_{0}^{\prime}).$ The quiver morphism $\theta:\tilde
{Q}\rightarrow Q^{\prime}$ is a continuous map on the topological
realizations. Therefore by e.g. \cite[Ch. 5, Lemma 6.7]{Massey91} it is a
topological covering. It is immediately seen to be a quiver morphism as well,
as the condition $F^{\prime}\theta=F$ forces $\theta$ to behave well on
vertices and arrows. Next, we have $F_{\ast}^{\prime}\theta_{\ast}(\pi
_{1}(\tilde{Q},\tilde{x}_{0}))=F_{\ast}(\pi_{1}(\tilde{Q},\tilde{x}_{0}))$ is
a normal subgroup of $\pi_{1}(Q,x_{0})$; also we have%

\[
F_{\ast}^{\prime}\theta_{\ast}(\pi_{1}(\tilde{Q},\tilde{x}_{0}))\subseteq
F_{\ast}^{\prime}(\pi_{1}(Q^{\prime},x_{0}^{\prime}))\subseteq\pi_{1}%
(Q,x_{0})
\]
from which we conclude that $\theta_{\ast}(\pi_{1}(\tilde{Q},\tilde{x}_{0}))$
is normal in $\pi_{1}(Q^{\prime},x_{0}^{\prime}).$ Thus $\theta$ is a Galois
covering of quivers. Furthermore, note that by Proposition \ref{min} $F_{\ast
}(N(\tilde{B},\tilde{x}_{0}))=N(B,x_{0})=F_{\ast}^{\prime}(N(B^{\prime}%
,x_{0}^{\prime}))$ and $N(B,x_{0})\subseteq F_{\ast}(\pi_{1}(\tilde{Q}%
,\tilde{x}_{0}))=F_{\ast}^{\prime}\theta_{\ast}(\pi_{1}(\tilde{Q},\tilde
{x}_{0})).$ We deduce that $N(B^{\prime},x_{0}^{\prime})\subseteq\theta_{\ast
}(\pi_{1}(\tilde{Q},\tilde{x}_{0})).$ By Theorem \ref{Cov} again we see that
$\theta$ is a Galois coalgebra covering, noting that $\tilde{B}$ is the span
of all liftings of minimal elements and paths in $B^{\prime}.$
\end{proof}

\section{Coverings from Gradings}

\begin{proposition}
\label{smash=cov}Let $B\subseteq\Bbbk Q$ be a pointed coalgebra and let
$\delta:Q_{1}\rightarrow G$ be a connected arrow weighting determining a
grading of $B$. Then $F:Q\rtimes G$ $\rightarrow Q$ is a Galois covering of
quivers and the restriction $F:\tilde{B}\rightarrow B$ is a coalgebra covering.
\end{proposition}

\begin{proof}
In view of Proposition \ref{Prop1} and by Theorem \ref{Cov} we need to show
that $N(B,x_{0})\subseteq F_{\ast}(\pi_{1}(Q\rtimes G,x_{0}\rtimes1)).$ Let
$[w^{-1}q^{-1}pw]\in N(B,x_{0})$ be a generator where $\sum\lambda_{i}p_{i}\in
B$ is a homogeneous minimal element with distinct paths $p_{1}=p$ and
$p_{2}=q$, both in $Q(x,y),$ and walk $w$ from $x_{0}$ to $x.$ Then, since $p$
and $q$ have the same weight, $u=w^{-1}q^{-1}pw$ is a closed walk in $Q$
having weight $1_{G}.$ It follows that $u$ lifts to a closed walk
$\widetilde{u}$ in $Q\rtimes G$ starting and ending at $x_{0}\rtimes1.$ We
have shown that $F_{\ast}([\widetilde{u}])=[u]\in F_{\ast}(\pi_{1}(Q\rtimes
G,x_{0}\rtimes1))$ and thus $N(B,x_{0})\subseteq F_{\ast}(\pi_{1}(Q\rtimes
G,x_{0}\rtimes1)).$
\end{proof}

\section{Liftings, weightings and isomorphisms\label{LVI}}

Let $F:\tilde{Q}\rightarrow Q$ be a Galois covering$.$ Let $L,L^{\prime}%
:Q_{0}\rightarrow\tilde{Q}_{0}$ be liftings of $F.$ By Theorem \ref{CSM} there
are covering isomorphisms with $\Bbbk Q\rtimes\Bbbk G\cong\Bbbk\tilde{Q}%
\cong\Bbbk Q\rtimes^{\prime}\Bbbk G$ determined by the liftings ($\rtimes
=\rtimes_{L},\rtimes^{\prime}=\rtimes_{L^{\prime}}$), so we can identify
$\tilde{Q}$ with the smash coproduct quiver $Q\rtimes G$ and then write
\[
L^{\prime}(x)=x\rtimes\gamma(x)
\]
for all $x\in Q_{0}$ where $\gamma:Q_{0}\rightarrow G$ is a function,
depending on $L^{\prime}$, that we call a \textit{vertex weighting}. We write
$\delta$ (resp. $\delta^{\prime})$ for the grading corresponding to $L$ (resp.
$L^{\prime}).$

Given a vertex weighting $\gamma,$ let $\delta^{\gamma}:Q_{1}\rightarrow G$
associate the $\gamma$\textit{-twisted grading} defined by the weight
function
\[
\delta^{\gamma}(a)=\gamma(y)^{-1}\delta(a)\gamma(x)
\]
for all all arrows $a\in Q_{1}(x,y)$, $x,y\in Q_{0}.$ Observe that this
formula extends to all paths (playing the role of the arrow $a)$, i.e., if
$p=a_{n}a_{n-1}\cdots a_{2}a_{1}$ is a path in $Q$ with vertex sequence
$t(a_{n})=x_{n},...,x_{1}=s(a_{2})=t(a_{1}),x_{0}=s(a_{0})$, then we set
$g_{i}=\gamma(x_{i})$ and define $\delta^{\gamma}(p)=\delta(a_{n}%
)\delta(a_{n-1})\cdots\delta(a_{2})\delta(a_{1}).$ We see that
\begin{align*}
\delta^{\gamma}(p)  &  =g_{n}^{-1}\delta(a_{n})g_{n-1}^{-1}g_{n-1}%
\delta(a_{n-1})\cdots g_{2}^{-1}g_{2}\delta(a_{2})g_{1}^{-1}g_{1}\delta
(a_{1})g_{0}\\
&  =\gamma(t(a_{n}))^{-1}\delta(a_{n})\delta(a_{n-1})\cdots\delta(a_{2}%
)\delta(a_{1})\gamma(s(a_{1}))\\
&  =\gamma(t(p))^{-1}\delta(p)\gamma(s(p)).
\end{align*}
An isomorphism $\theta:F\rightarrow F^{\prime}$ of Galois coverings of $Q$ (or
of the coalgebras coverings $\Bbbk\tilde{Q}\overset{F}{\rightarrow}\Bbbk
Q\overset{F^{\prime}}{\leftarrow}\Bbbk Q^{\prime})$ is said to be a
$G$\textit{-isomorphism} if it commutes with the right action of the
automorphism group
\[
G=\frac{\pi_{1}(Q,x)}{F_{\ast}(\pi_{1}(\tilde{Q},\tilde{x}))}=\frac{\pi
_{1}(Q,x)}{F_{\ast}^{\prime}(\pi_{1}(Q^{\prime},x^{\prime}))}%
\]
where $F(\tilde{x})=F^{\prime}(x^{\prime})=x.$ Let $\Bbbk Q\rtimes^{\prime
}\Bbbk G$ be be a smash coproduct coalgebra using a grading $\delta^{\prime
}:Q_{1}\rightarrow G.$ For a $G$-isomorphism of smash coproduct
coverings\textit{ }$\theta:\Bbbk Q\rtimes\Bbbk G\rightarrow\Bbbk
Q\rtimes^{\prime}\Bbbk G$ we have $\theta(p\rtimes g)=p\rtimes^{\prime}g_{p}g$
for some $g_{p}\in G,$ for all paths $p$ and $g\in G.$ We say that the
weighting $\delta^{\prime}$ implicit in the smash coproduct $\Bbbk
Q\rtimes^{\prime}\Bbbk G$ is the $G$\textit{-grading associated to the }%
$G$\textit{-isomorphism }$\theta.$

Given the associated gradings to each lifting, vertex weighting and
$G$-isomorphism $\theta$ we have described, we have the following result.

\begin{proposition}
Let $F:\tilde{Q}\rightarrow Q$ be a Galois covering of quivers and fix a
lifting $L:Q_{0}\rightarrow\tilde{Q}_{0}.$ There are bijections between the
following sets:\newline(a) liftings $L^{\prime}:Q_{0}\rightarrow\tilde{Q}_{0}%
$\newline(b) vertex weightings $\gamma:Q_{0}\rightarrow G$\newline(c)
$G$-isomorphisms of coverings $\Bbbk Q\rtimes\Bbbk G\rightarrow\Bbbk
Q\rtimes^{\prime}\Bbbk G$ \newline Moreover, these bijections preserve the
associated gradings.
\end{proposition}

\begin{proof}
By Theorem \ref{CSM} we know that $\Bbbk\tilde{Q}\cong\Bbbk Q\rtimes\Bbbk G$
where $\rtimes$ is the smash coproduct with grading induced by $L.$ Let
$L^{\prime}:Q_{0}\rightarrow\tilde{Q}_{0}$ be a lifting. By the isomorphism,
we may assume $\Bbbk\tilde{Q}=\Bbbk Q\rtimes\Bbbk G$ with $\tilde{Q}=Q\rtimes
G$ and therefore $L^{\prime}(x)=x\rtimes\gamma(x)$ for all $x\in Q_{0}$ for
some $\gamma(x)\in G.$ This produces a vertex lifting $\gamma.$ Since each
such choice of $\gamma$ provides a unique lifting, we have obtained a
bijection between the sets in (a) and (b). Let $\delta^{\prime}:Q_{1}%
\rightarrow G$ be the arrow weighting induced by $L^{\prime}.$ Let $p$ be a
path in $Q(x,y).$ Then $L^{\prime}(p)=p\rtimes\gamma(x)$ is a path in
$\tilde{Q}$ starting at $x\rtimes\gamma(x)$ and ending at%
\begin{align*}
&  y\rtimes\delta(p)\gamma(x)\\
&  =y\rtimes\gamma(y)\gamma(y)^{-1}\delta(p)\gamma(x)\\
&  =L^{\prime}(y)^{\gamma(y)^{-1}\delta(p)\gamma(x)}\\
&  =L^{\prime}(y)^{\delta^{\gamma}(p)}.
\end{align*}
This shows that $\delta^{\prime}$ $=\delta^{\gamma}$, so that the associated
grading is preserved, as claimed.

We move on to demonstrating a grading-preserving isomorphism between (b) and
(c). For any vertex weighting $\gamma:Q_{0}\rightarrow G$, define a map
$\theta_{\gamma}:\Bbbk Q\rtimes\Bbbk G\rightarrow\Bbbk Q\rtimes^{\prime}\Bbbk
G$ by $\theta_{\gamma}(p\rtimes g)=p\rtimes^{\prime}\gamma(x)^{-1}g$ for all
paths $p$ and $g\in G.$ This clearly defines a linear isomorphism with
$F^{\prime}\theta=F$. On the other hand, given a $\Bbbk$-linear isomorphism
$\theta:\Bbbk Q\rtimes\Bbbk G\rightarrow\Bbbk Q\rtimes^{\prime}\Bbbk G$ such
that for all $p\rtimes g\in\Bbbk Q\rtimes\Bbbk G$ for all paths $p$ in $Q$ and
$g\in G$, $\theta(p\rtimes g)=p\rtimes g_{p}g$ for some $g_{p}\in G$, we note
that
\begin{align*}
\Delta(\theta(p\rtimes1))  &  =\Delta(p\rtimes^{\prime}g_{p})\\
&  =\sum_{p=rq}(r\rtimes^{\prime}\delta^{\prime}(q)g_{p})\otimes
(q\rtimes^{\prime}g_{p})
\end{align*}
and on the other hand
\begin{align*}
(\theta\otimes\theta)\Delta(p\rtimes1)  &  =(\theta\otimes\theta)\sum
_{p=rq}(r\rtimes\delta(q))\otimes(q\rtimes1)\\
&  =\sum_{p=rq}(r\rtimes^{\prime}g_{r}\delta(q))\otimes(q\rtimes^{\prime}%
g_{q}).
\end{align*}
Equating the right tensor factors yields $g_{p}=g_{q}$ for all initial
segments $q$ of $p.$ In particular, we see in this situation that $g_{p}$ is
determined by the starting vertex, i.e., $g_{p}=g_{s(p)}$ for all paths $p.$
Let the vertex weighting $\gamma$ be defined by $\gamma(x)=g_{x}^{-1}$ for all
$x\in Q_{0}.$ Next, equating the left tensor factors at $p=q$ results in the
equation $\delta^{\prime}(p)=g_{t(p)}\delta(p)g_{p}^{-1}=\gamma(y)^{-1}%
\delta(p)\gamma(x)=\delta^{\gamma}(p)$ for all $p\in Q(x,y).$ This shows that
$\theta$ is a $G$-coalgebra covering isomorphism if and only if $\delta
^{\prime}=\delta^{\gamma}$ and $\theta=\theta_{\gamma}. $ It follows that we
have a bijective mapping from the set (c) to the set (b) given by
$\gamma\longmapsto\theta_{\gamma},$ which preserves the associated grading.
\end{proof}

\begin{remark}
The set of vertex weightings $G^{Q_{0}}$ forms group under pointwise
multiplication in $G.$ Therefore we can construe the bijections in the Theorem
as group isomorphisms, and the group of gradings (again pointwise) as a
homomorphic image of each of the three isomorphic groups in (a)-(c). The quite
arbitrary choice of the lifting $L$ provides an identity element in the group
of liftings, corresponding to the neutral vertex weighting $x\mapsto1_{G}$,
$x\in Q_{0}.$
\end{remark}

\section{Universality}

\begin{theorem}
\label{Univ}Let $B\subseteq\Bbbk Q$ be a pointed coalgebra. Then \newline(a)
there exists a Galois coalgebra covering $F:\tilde{B}\rightarrow B$ such that
for every Galois coalgebra covering $F^{\prime}:B^{\prime}\rightarrow B$,
there exists a Galois coalgebra covering $E:\tilde{B}\rightarrow B^{\prime}$
such that the following diagram commutes.%
\[%
\begin{array}
[c]{ccccc}%
\tilde{B} &  & \overset{E}{\xrightarrow{\hspace*{.75cm}}} &  & B^{\prime}\\
& \underset{F}{\searrow} &  & \underset{F^{\prime}}{\swarrow} & \\
&  & B &  &
\end{array}
\]
\newline\newline(b) If we fix base points $x_{0}\in B_{0},$ $x_{0}^{\prime}\in
B_{0}^{\prime},$ $\tilde{x}_{0}\in\tilde{B}_{0}$ with $F(\tilde{x}_{0}%
)=x_{0}=F^{\prime}(x_{0}^{\prime}),$ then $E$ can be uniquely chosen so that
$E(\tilde{x}_{0})=x_{0}^{\prime}.$ \newline(c) The Galois covering
$F:\tilde{B}\rightarrow B$ is unique up to isomorphism.
\end{theorem}

\begin{proof}
Fix $x_{0}\in Q_{0}.$ The smash coproduct $Q\rtimes\pi_{1}(Q,x_{0})$ is the
universal covering of $Q$ where the connected grading $\hat{\delta}%
:Q_{1}\rightarrow\pi_{1}(Q,x_{0})$ arises from a lifting $L$ such that
$L(x_{0})=x_{0}\rtimes1_{\pi_{1}(Q,x_{0})}.$ Let $G=\pi_{1}(Q,x_{0}),$
$N=N(B,x_{0})$ and $\tilde{Q}=Q\rtimes G/N,$ where the grading is induced by
composing with the natural map onto $G/N$, i.e., $\delta$ is the composition
\[
Q_{1}\overset{\hat{\delta}}{\rightarrow}G\rightarrow G/N
\]
This is immediately seen to give a connected grading of $\Bbbk Q$. Now observe
that $F_{\ast}(\pi_{1}(\tilde{Q},x_{0}\rtimes1))=N,$ as the elements of $N$
are precisely the equivalence classes of closed walks based at $x_{0}$ that
lift to closed walks at $x_{0}\rtimes1.$ By Theorem \ref{Cov}, we obtain a
coalgebra covering $F:\tilde{B}\rightarrow B$ where $\tilde{B}=B\rtimes G$ is
the span of liftings of minimal relations of $B.$

Let $F^{\prime}:B^{\prime}\rightarrow B$ be the hypothetical coalgebra
covering arising as the restriction of a covering $F^{\prime}:Q^{\prime
}\rightarrow Q$ with base vertex $x_{0}^{\prime}\in Q_{0}^{\prime},$
$F^{\prime}(x_{0}^{\prime})=x_{0}$. Let $N^{\prime}=F_{\ast}^{\prime
}(N(B^{\prime},x_{0}^{\prime}))$. By Theorem \ref{Cov} we have
\[
N^{\prime}\supseteq N=F_{\ast}(N(\pi_{1}(\tilde{Q},\tilde{x}_{0})))
\]
Thus we may form the smash coproduct $Q\rtimes G/N^{\prime}$ using the induced
connected grading via $Q_{1}\overset{\delta}{\rightarrow}G\rightarrow
G/N^{\prime}.$ Here $F_{\ast}^{\prime}\left(  \pi_{1}(Q\rtimes G/N^{\prime
},x_{0}\rtimes1\right)  )=N^{\prime}=F_{\ast}^{\prime}(N(B^{\prime}%
,x_{0}^{\prime})),$ so by standard results e.g. \cite[Chapter 5, Cor.
6.4]{Massey91}, there is a unique isomorphism of coverings $E:Q\rtimes
G/N^{\prime}\tilde{\rightarrow}Q^{\prime}$ sending $x_{0}\rtimes1$ to
$x_{0}^{\prime}\footnote{If we identify $Q^{\prime}$ with $Q\rtimes
G/N^{\prime}$ and $x_{0}^{\prime}=x\rtimes\bar{g}$ with $\bar{g}\in G/N,$ then
the isomorphism is concretely given by the right action of $\bar{g}.$}.$ We
have the commuting diagram of quivers%
\[%
\begin{array}
[c]{ccccc}%
Q\rtimes G/N & \longrightarrow & Q\rtimes G/N^{\prime} & \longrightarrow &
Q^{\prime}\\
& \underset{F}{\searrow} &  & \underset{F^{\prime}}{\swarrow} & \\
&  & Q &  &
\end{array}
\]
The horizontal composed map $Q\rtimes G/N\rightarrow Q^\prime$ provides a
morphism of coalgebra coverings by Lemma \ref{morph}$.$ This proves the
assertions (a) and (b).

The uniqueness of $F$ follows since it is the unique Galois covering such that
$F_{\ast}\left(  \pi_{1}(\tilde{Q},x_{0}\rtimes1\right)  )=N$.
\end{proof}

\bigskip

The covering coalgebra $Q\rtimes G/N$ in this result is the \textit{universal
Galois covering of} $B\subseteq kQ$.

\section{Examples}

\begin{example}
\label{loop}Let $Q$ be the quiver consisting of a single loop $a$ and single
vertex $x.$ The universal cover $\tilde{Q}$ is a quiver of type $\mathbb{A}%
_{\infty}$ with all arrows in the same direction. The automorphism group
$\pi_{1}(Q,x)$ is infinite cyclic, generated by $g=[a]$. A connected grading
is given by $\delta(a)=g.$ Since there is a single vertex and $G$ is abelian,
Theorem \ref{LVI} says that all other liftings yield the same grading. The
path $i\rightarrow\cdot\rightarrow\cdots\cdot\rightarrow i+\ell,$
$i\in\mathbb{Z},$ $\ell\in\mathbb{N}$ corresponds to $a^{\ell}\rtimes g^{i}$
$\in\Bbbk Q\rtimes G.$ All other coverings of $Q$ are given by the action of a
subgroup $<g^{n}>$ of $G$ and are easily seen to be the cyclic quiver $\bar
{0}\rightarrow\bar{1}\rightarrow\cdots\rightarrow\overline{n-1}\rightarrow
\bar{n}=\bar{0}$ of length $n\in\mathbb{N}$. The only subcoalgebras of $\Bbbk
Q$ are the truncations $B=\Bbbk\{x,a,a^{2}\cdots a^{n-1}\}$. There are no
minimal elements, so the universal covering $\tilde{B}$ is isomorphic to
$B\rtimes G$ where the path $i\rightarrow\cdot\rightarrow\cdots\cdot
\rightarrow i+\ell$ corresponds to $a^{\ell}\rtimes g^{i},$ $0\leq\ell\leq
n,\quad i\in\mathbb{Z}$. Each finite-dimensional comodule for $\Bbbk\tilde{Q}$
corresponds to a quiver representation $\Bbbk\rightarrow\Bbbk\rightarrow
\cdots\cdot\rightarrow\Bbbk,$ which pushes down to the $\ell$-dimensional
representation of $Q$ corresponding to the comodule $\Bbbk\{x,a,a^{2}\cdots
a^{\ell-1}\}$. Since these comodules are precisely the representatives of
finite-dimensional indecomposables for $\Bbbk Q$, it is clear that the
forgetful functor $\mathcal{M}^{k\tilde{Q}}\approx\mathrm{Gr}^{\Bbbk
Q}\rightarrow\mathcal{M}^{kQ}$ is dense. We note here that the indecomposable
representations of $Q$ corresponding to indecomposables over the path algebra
$k[a,a^{-1}]$ with nonzero (Jordan) eigenvalue are not comodules as they are
not locally nilpotent (cf. \cite{Chin2004}).
\end{example}

\begin{example}
\label{free}Let $Q$ be the quiver consisting of two loops $a,b$ and single
vertex $x.$ The fundamental group $G$ is a free group on two generators. The
quiver $\tilde{Q}$ is the Cayley graph of the free group on $a,b$ and the
vertices of $\tilde{Q}$ are indexed by the elements of $G,$ and these group
elements correspond to vertex weightings. Distinct vertex weightings give rise
to distinct gradings, which are thus infinite in number.
\end{example}

\begin{example}
\label{Kron}Let $Q$ be the Kronecker quiver
\[
x\rightrightarrows y
\]
consisting of the two arrows $a,b$ from the vertex $x$ to the vertex $y$. The
fundamental group is infinite cyclic. Specifying an arrow weighting
$\delta:Q_{0}\rightarrow G$ by $\delta(a)=1$ and $\delta(b)=g,$ we get the
covering quiver $\tilde{Q}=Q\rtimes G$ of type $\mathbb{A}_{\infty}$ with
zig-zag orientation%
\[
\cdots\leftarrow x_{0}\rightarrow y_{0}\leftarrow x_{1}\rightarrow
y_{1}\leftarrow x_{2}\rightarrow\cdots
\]
where $x_{n}=x\rtimes g^{n}$ and $y_{n}=y\rtimes g^{n}.$ Here again there are
infinitely many distinct gradings, with isomorphic smash coproducts. The
finite-dimensional indecomposable comodules for $\Bbbk\tilde{Q}$ are given by
the representations of $\tilde{Q}.$ By the theory of special biserial
(co)algebras, see \cite{Chin2010} (and \cite{Erdmann90}), these
representations are given by the strings $\Bbbk-\Bbbk-\cdots-\Bbbk-\Bbbk$ of
finite length where $-$ denotes $\leftarrow$ or $\rightarrow.$ On the other
hand, the finite-dimensional indecomposable representations of $Q$ are
well-known to be given by string modules and a one-parameter family of band
modules \cite{Erdmann90}. Note that the band modules correspond to comodules
in $\mathcal{M}^{B},$ as they are locally nilpotent (cf. \cite{Chin2004}), in
contrast to Example \ref{loop}. The band comodules correspond to non-gradable
comodules for $\Bbbk Q.$ Thus the forgetful functor $\mathcal{M}^{\tilde{B}%
}\approx\mathrm{Gr}^{B}\rightarrow\mathcal{M}^{B}$ is not dense.
\end{example}

\begin{example}
\label{SL2}Consider the coordinate Hopf algebra $\Bbbk_{\zeta}[SL(2)]$ at a
root of unity $\zeta$ of odd order $\ell$ over a field $\Bbbk$ of
characteristic zero. The basic coalgebra decomposes in to block coalgebras
$B_{r}$, $r=0,1,2$, $\cdots,\ell-2$. The nontrivial blocks are indexed by
integers $r\leq\ell-2,$ and they are all isomorphic \cite{Chin2010}. Each
nontrivial block $B$ is isomorphic to the subcoalgebra of path coalgebra of
the quiver $Q:$
\[
x_{0}\overset{b_{0}}{\underset{a_{0}}{\leftrightarrows}}x_{1}\overset{b_{1}%
}{\underset{a_{1}}{\leftrightarrows}}x_{2}\overset{b_{2}}{\underset{a_{2}%
}{\leftrightarrows}}\cdot\cdot\cdot
\]
spanned by the by group-likes $x_{i}$ corresponding to vertices, arrows
$a_{i},b_{i},$ $i\geq0,$ together with coradical degree two elements
\begin{align*}
d_{0} &  :=b_{0}a_{0}\\
d_{i+1} &  :=a_{i}b_{i}+b_{i+1}a_{i+1},\text{ }i\geq0.
\end{align*}
Therefore $N(B,x_{0})$ is generated by homotopy classes closed walks of the
form $w^{-1}b_{i}^{-1}a_{i}^{-1}b_{i+1}a_{i+1}w$ for appropriate walks $w$
(from $x_{0}$ to $x_{i}$), and it follows that the fundamental group of
$B\subset\Bbbk Q$ is the infinite cyclic group $G=<g>$, generated by
$g=[b_{0}a_{0}]$ and the universal cover is a quiver of type $\mathbb{ZA}%
_{\infty}.$ For example letting $\delta(b)=g^{-1}$ and $\delta(a)=1_{G}$ we
obtain a connected grading of $B$ and covering quiver $Q\rtimes G$%
\[%
\begin{array}
[c]{cccccc}
&  & \vdots &  &  & \\
x_{0,1} & \overset{a_{01}}{\xrightarrow{\hspace*{.4cm}}} & x_{1,1} &
\overset{}{\xrightarrow{\hspace*{.4cm}}} & x_{2,1} & \\
& \swarrow &  & \overset{}{\swarrow} &  & \\
x_{0,0} & \xrightarrow{\hspace*{.4cm}} & x_{1,0} &
\xrightarrow{\hspace*{.4cm}} & x_{2,0} & \cdots\\
& \swarrow &  & \overset{b_{1,-1}}{\swarrow} &  & \\
x_{0,-1} & \xrightarrow{\hspace*{.4cm}} & x_{1,-1} &
\xrightarrow{\hspace*{.4cm}} & x_{2,1} & \\
&  & \vdots &  &  &
\end{array}
\]
with vertices $x_{in}=x_{i}\rtimes g^{n};$ $i\in\mathbb{N},$ $n\in\mathbb{Z}.$
The arrows are $a_{in}=a_{i}\rtimes g^{n}$ starting at $x_{in}$ and ending at
$x_{i+1,n},$ and $b_{in}=b_{i}\rtimes g^{n}$ starting at $x_{i+1,n+1}\rtimes
g^{n+1}$ and ending at $x_{in}.$ The coalgebra $\tilde{B}$ is spanned by the
vertices, arrows, paths $d_{0n}:=b_{0,n-1}a_{0,n}$ and the minimal elements
$d_{i+1,n}:=a_{in}b_{in}+b_{i+1,n}a_{i+1,n+1},$ $i\geq0.$\newline The finite
dimensional representations of $B$ are determined in \cite{Chin2010} . Each
arrow $a_{i}$ generates the length two right comodule known as a \textit{Weyl
comodule}, with composition series $%
\genfrac{.}{.}{0pt}{}{\Bbbk x_{i}}{\Bbbk x_{i+1}}%
.$ The arrow $b_{i}$ generates the dual Weyl comodule. Any grading of $B$ is
determined by an arrow weighting, so each Weyl and dual Weyl comodule is
obviously thus graded. The coalgebra $B$ is an example of a special biserial
coalgebra and all finite dimensional comodules are string comodules. Each of
these comodules correspond to a walk (i.e. a string) of one of the following
forms
\begin{align*}
&  a_{t}b_{t-1}^{-}...a_{s-1}b_{s+1}^{-}a_{s}\\
&  b_{t}^{-}a_{t-1}...b_{s-1}a_{s+1}b_{s}^{-}\\
&  a_{t}b_{t-1}^{-}...a_{s+1}b_{s}^{-}\\
&  b_{t}^{-}a_{t-1}...b_{s+1}^{-}a_{s}%
\end{align*}
where the subscripts form an interval $[s,t]$ of nonnegative integers strictly
increasing from right to left in the walk . Each of these can be constructed
as an iteration of pullbacks and pushouts of Weyl comodules and dual Weyl
comodules (with isomorphic socles or tops), both of which are gradable. Since
the simple comodules have trivial weighting, it follows that every
finite-dimensional $B$-comodule is gradable. Thus the forgetful functor
$\mathrm{Gr}^{B}$ $\rightarrow\mathcal{M}^{B}$ is dense.
\end{example}

\bibliographystyle{plain}
\bibliography{chinsrefs}

\begin{thebibliography}{10}

\bibitem{BG81}
K.~Bongartz and P.~Gabriel.
\newblock Covering spaces in representation-theory.
\newblock {\em Invent. Math.}, 65(3):331--378, 1981/82.

\bibitem{Chin2002}
William Chin.
\newblock Hereditary and path coalgebras.
\newblock {\em Comm. Algebra}, 30(4):1829--1831, 2002.

\bibitem{Chin2004}
William Chin.
\newblock A brief introduction to coalgebra representation theory.
\newblock In {\em Hopf algebras}, volume 237 of {\em Lecture Notes in Pure and
  Appl. Math.}, pages 109--131. Dekker, New York, 2004.

\bibitem{Chin2010}
William Chin.
\newblock Representations of quantum sl(2) at roots of 1.
\newblock {\em preprint}, 2009.

\bibitem{CKQ2002}
William Chin, Mark Kleiner, and Declan Quinn.
\newblock Almost split sequences for comodules.
\newblock {\em J. Algebra}, 249(1):1--19, 2002.

\bibitem{ChinMontg}
William Chin and Susan Montgomery.
\newblock Basic coalgebras.
\newblock In {\em Modular interfaces ({R}iverside, {CA}, 1995)}, volume~4 of
  {\em AMS/IP Stud. Adv. Math.}, pages 41--47. Amer. Math. Soc., Providence,
  RI, 1997.

\bibitem{CibilsMarcos06}
Claude Cibils and Eduardo~N. Marcos.
\newblock Skew category, {G}alois covering and smash product of a
  {$k$}-category.
\newblock {\em Proc. Amer. Math. Soc.}, 134(1):39--50 (electronic), 2006.

\bibitem{DNRV}
S.~D{\u{a}}sc{\u{a}}lescu, C.~N{\u{a}}st{\u{a}}sescu, S.~Raianu, and
  F.~Van~Oystaeyen.
\newblock Graded coalgebras and {M}orita-{T}akeuchi contexts.
\newblock {\em Tsukuba J. Math.}, 19(2):395--407, 1995.

\bibitem{Erdmann90}
Karin Erdmann.
\newblock {\em Blocks of tame representation type and related algebras}, volume
  1428 of {\em Lecture Notes in Mathematics}.
\newblock Springer-Verlag, Berlin, 1990.

\bibitem{Gabriel81}
P.~Gabriel.
\newblock The universal cover of a representation-finite algebra.
\newblock In {\em Representations of algebras ({P}uebla, 1980)}, volume 903 of
  {\em Lecture Notes in Math.}, pages 68--105. Springer, Berlin, 1981.

\bibitem{Green83}
Edward~L. Green.
\newblock Graphs with relations, coverings and group-graded algebras.
\newblock {\em Trans. Amer. Math. Soc.}, 279(1):297--310, 1983.

\bibitem{GrossTucker}
Jonathan~L. Gross and Thomas~W. Tucker.
\newblock {\em Topological graph theory}.
\newblock Wiley-Interscience Series in Discrete Mathematics and Optimization.
  John Wiley \& Sons Inc., New York, 1987.
\newblock A Wiley-Interscience Publication.

\bibitem{Hatcher}
Allen Hatcher.
\newblock {\em Algebraic topology}.
\newblock Cambridge University Press, Cambridge, 2002.

\bibitem{Simson05}
Justyna Kosakowska and Daniel Simson.
\newblock Hereditary coalgebras and representations of species.
\newblock {\em J. Algebra}, 293(2):457--505, 2005.

\bibitem{LeMeur07}
Patrick Le~Meur.
\newblock The universal cover of an algebra without double bypass.
\newblock {\em J. Algebra}, 312(1):330--353, 2007.

\bibitem{Martinez83}
R.~Mart{\'{\i}}nez-Villa and J.~A. de~la Pe{\~n}a.
\newblock The universal cover of a quiver with relations.
\newblock {\em J. Pure Appl. Algebra}, 30(3):277--292, 1983.

\bibitem{Massey91}
William~S. Massey.
\newblock {\em A basic course in algebraic topology}, volume 127 of {\em
  Graduate Texts in Mathematics}.
\newblock Springer-Verlag, New York, 1991.

\bibitem{Nichols78}
Warren~D. Nichols.
\newblock Bialgebras of type one.
\newblock {\em Comm. Algebra}, 6(15):1521--1552, 1978.

\bibitem{NowakSimson2002}
Sebastian Nowak and Daniel Simson.
\newblock Locally {D}ynkin quivers and hereditary coalgebras whose left
  comodules are direct sums of finite dimensional comodules.
\newblock {\em Comm. Algebra}, 30(1):455--476, 2002.

\bibitem{Riedtmann80}
C.~Riedtmann.
\newblock Algebren, {D}arstellungsk\"ocher, \"{U}berlagerungen und zur\"uck.
\newblock {\em Comment. Math. Helv.}, 55(2):199--224, 1980.

\bibitem{Simson01}
Daniel Simson.
\newblock Coalgebras, comodules, pseudocompact algebras and tame comodule type.
\newblock {\em Colloq. Math.}, 90(1):101--150, 2001.

\bibitem{Simson04}
Daniel Simson.
\newblock Path coalgebras of quivers with relations and a tame-wild dichotomy
  problem for coalgebras.
\newblock In {\em Rings, modules, algebras, and abelian groups}, volume 236 of
  {\em Lecture Notes in Pure and Appl. Math.}, pages 465--492. Dekker, New
  York, 2004.

\bibitem{Simson07}
Daniel Simson.
\newblock Path coalgebras of profinite bound quivers, cotensor coalgebras of
  bound species and locally nilpotent representations.
\newblock {\em Colloq. Math.}, 109(2):307--343, 2007.

\bibitem{Woodcock97}
D.~Woodcock.
\newblock Some categorical remarks on the representation theory of coalgebras.
\newblock {\em Comm. Algebra}, 25(9):2775--2794, 1997.

\end{thebibliography}

\end{document}